\documentclass[10pt]{amsart}

\usepackage{amsmath,amssymb,latexsym,cite,mathrsfs}
\usepackage[utf8]{inputenc}
\usepackage{color,enumitem,graphicx}
\usepackage[colorlinks=true,urlcolor=blue,
citecolor=red,linkcolor=blue,linktocpage,pdfpagelabels,
bookmarksnumbered,bookmarksopen]{hyperref}
\usepackage[english]{babel}

\usepackage[left=2.9cm,right=2.9cm,top=2.8cm,bottom=2.8cm]{geometry}
\usepackage[hyperpageref]{backref}

\usepackage[colorinlistoftodos]{todonotes}
\makeatletter
\providecommand\@dotsep{5}
\def\listtodoname{List of Todos}
\def\listoftodos{\@starttoc{tdo}\listtodoname}
\makeatother

\numberwithin{equation}{section}

\everymath{\displaystyle}

\newtheorem{theorem}{Theorem}[section]

\newtheorem{lemma}[theorem]{Lemma}

\newtheorem{corollary}[theorem]{Corollary}

\newtheorem{remark}{Remark}[section]

\newcommand{\R}{\mathbb{R}}
\newcommand{\N}{\mathbb{N}}

\begin{document}

\title{A NONLOCAL LOGISTIC EQUATION WITH NONLINEAR ADVECTION TERM}
\author{Romildo N. de Lima, Ronaldo C. Duarte and Marco A. S. Souto}
\address[Romildo N. de Lima and Marco A. S. Souto]
{\newline\indent Unidade Acadêmica de Matemática
\newline\indent 
Universidade Federal de Campina Grande 
\newline\indent
58429-970, Campina Grande - PB, Brazil} 
\email{\href{mailto: romildo@mat.ufcg.edu; marco@mat.ufcg.edu}{romildo@mat.ufcg.edu; marco@mat.ufcg.edu}}

\address[Ronaldo C. Duarte]
{\newline\indent Departamento de Matemática
	\newline\indent 
	Universidade Federal do Rio Grande do Norte 
	\newline\indent
	59078-970, Natal - RN, Brazil} 
\email{\href{mailto: ronaldo.cesar.duarte@ufrn.br}{ronaldo.cesar.duarte@ufrn.br}}

\pretolerance10000


\begin{abstract} 
In this paper, we study a nonlocal logistic equation with nonlinear advection term
$$
\left\{
\begin{array}{lcl}
-\Delta u+\vec{\alpha}(x)\cdot \nabla (|u|^{p-1}u)=\left(\lambda-\int_{\Omega}K(x,y)u^{\gamma}(y)dy \right)u,\mbox{ in }\Omega\\
u=0,\mbox{ on }\partial\Omega
\end{array}
\right.\leqno(P)_p
$$
where $\Omega\subset\R^N$, $N\geq1$, is a bounded domain with smooth boundary, $\vec{\alpha}(x)=(\alpha_1(x),\cdots,\alpha_N(x))$ is a flow satisfying suitable condition, $\gamma>0$, $p\geq1$, $\lambda\in\R$ and $K:\Omega\times\Omega\rightarrow\R$ is a nonnegative function with $K\in L^{\infty}(\Omega\times\Omega)$ and verifying other conditions that will be detailed below. It is very important to note that, this equation is not the classic logistic equation due to the inclusion of the term $\vec{\alpha}(x)\cdot \nabla (|u|^{p-1}u)$, moreover, the inclusion of the integral nonlocal term on the right-hand side makes the problem closer to a real world situation. 
\end{abstract}

\thanks{M. A. S. Souto was partially supported by CNPq/Brazil  309.692/2020-2 and grant 3031/2021 FAPESQ-PB} 
\thanks{R. N de Lima was partially supported by CNPq/Brazil 306.411/2022-9 and grant 3031/2021 and 3177/2021 FAPESQ-PB}

\maketitle
\tableofcontents
\section{Introduction and Main Results}

The main goal of this paper is to study the existence of a positive solution for the nonlocal logistic equation with nonlinear advection term
$$
\left\{
\begin{array}{lcl}
	-\Delta u+\vec{\alpha}(x)\cdot \nabla (|u|^{p-1}u)=\left(\lambda-\int_{\Omega}K(x,y)u^{\gamma}(y)dy \right)u,\mbox{ in }\Omega\\
	u=0,\mbox{ on }\partial\Omega
\end{array}
\right.\leqno(P)_p
$$
where $\Omega\subset\R^N$, $N\geq1$, is a bounded domain with smooth boundary, $\vec{\alpha}(x)=(\alpha_1(x),\cdots,\alpha_N(x))$ is a flow satisfying suitable condition, $\gamma>0$, $p\geq1$ and $\lambda\in\R$. 

The equation used to model the behavior of a species inhabiting a smooth bounded domain $\Omega\subset\R^N$, so that the border $\partial\Omega$ is considered a lethal region, it is  the classical logistic equation is given by

\begin{equation} \label{1}
	\left\{
	\begin{array}{lcl}
		-\Delta u=u(\lambda-b(x)u^{p}),\quad\mbox{in} \quad\Omega\\
		u=0,\quad\mbox{on}\quad\partial\Omega,
	\end{array}
	\right.
\end{equation}
where $u(x)$ is the population density at location $x\in\Omega$, $\lambda\in\R$ is the growth rate of the species, $b$ is a positive function denoting the carrying capacity, that is, $b(x)$ describes the limiting effect of crowding of the population and $p>0$.

Obviously (\ref{1}) is a local problem, the crowding effect of the population $u$ at $x$ depends only on the value of the population in the same point $x$. More realistic situations have already been considered, where the crowding effect also depends on the value of the population around $x$, that is, the crowding effect depends on the value of the integral involving the function $u$ over the ball $B_{r}(x)$ centered at $x$ of radius $r>0$. Precisely, in \cite{Chipot} the following nonlocal problem has been studied   
\begin{equation} \label{2}
	\left\{
	\begin{array}{lcl}
		-\Delta u=\left(\lambda-\int_{\Omega\cap B_{r}(x)}b(y)u^{p}(y)dy\right)u,\quad\mbox{in} \quad\Omega\\
		u=0,\quad\mbox{on}\quad\partial\Omega
	\end{array}
	\right.
\end{equation}
where $b$ is a nonnegative and nontrivial continuous function. After the paper \cite{Chipot}, great attention has been given for the problem 
\begin{equation} \label{3}
	\left\{
	\begin{array}{lcl}
		-\Delta u=\left(\lambda-\int_{\Omega}K(x,y)u^{p}(y)dy\right)u,\quad\mbox{in}\quad\Omega\\
		u=0,\quad\mbox{on}\quad\partial\Omega
	\end{array}
	\right.
\end{equation}
by supposing different conditions for $K$, see for example, Allegretto and Nistri \cite{Allegretto-Nistri}, Alves, Delgado, Souto and  Suárez \cite{Alves-Delgado-Souto-Suarez}, Chen and Shi \cite{Chen-Shi}, Corrêa, Delgado and Suárez \cite{Correa-Delgado-Suarez}, Coville \cite{Coville}, de Lima and Souto \cite{RM}, Leman, Méléard and Mirrahimi \cite{Leman-Meleard-Mirrahimi}, and Sun, Shi and Wang \cite{Sun-Shi-Wang} and other references. 

In addition to these equations, there are other equations and system that can model the behavior of specific or general species, for example Delgado, Duarte and Suárez in \cite{DDS} present a system arising from the amoeba-bacteria population dynamics; Cintra,  Morales-Rodrigo and Suárez in \cite{CMS} studied the existence and non-existence of coexistence states for a cross-diffusion system arising from a prey–predator model with a predator satiation term, and an other good reference is \cite{U}, where Umezu presents a logistic equation with nonlinear boundary condition arises from coastal fishery harvesting.

On the other hand, reaction-advection-diffusion equations can also be used to model phenomena in population dynamics. In this models it is considered the local rate of reproduction per individual (reaction term) and the spatial movement of the species which is random (diffusive) and directive (advective), see for example \cite{CantrellCosner} and \cite{Cosner}.

Cintra, Montenegro and Suárez, in \cite{CintraMontenegroSuarez}, consider the classical logistic with reaction term $f(\lambda,u)=\lambda u-u^2$, with random movement described by the Laplacian operator and, for $\vec{\alpha}(x)=\vec{\alpha}\in\R^N$, a nonlinear advection term $\vec{\alpha}\cdot \nabla u^p=pu^{p-1}\vec{\alpha}\cdot\nabla u$, that is
\begin{equation}\label{AUX}
\left\{
\begin{array}{lcl}
	-\Delta u+\vec{\alpha}\cdot \nabla u^p=\lambda u-u^2,\mbox{ in }\Omega\\
	u=0,\mbox{ on }\partial\Omega.
\end{array}
\right.
\end{equation}
As explained by Cintra, Montenegro and Suárez, the inclusion of this nonlinear term means that the species has a directed motion and that the rate at which the population moves up depends on $u$, $p$ and $|\vec{\alpha}|$. In particular, as $u=0$ on $\partial\Omega$, it becomes arbitrarily small as the species approaches the boundary. Thus, the situation studied by them is a more realistic model if compared with the case where the advection is linear. From the mathematical point of view, the inclusion of this term brings technical difficulties in the analysis, especially because it has no definite sign, these difficulties were overcome by Cintra, Montenegro and Suárez.

It is important to observe that when the advective term is present, there are fewer available results compared to the classic logistic case. Some of the few works related to this case are Cantrell and Cosner \cite{CantrellCosner2}, where the case $p=1$ was studied, Belgacem and Cosner \cite{BelgacemCosner} studied the following logistic equation with linear advection or drift term was analyzed
\begin{equation}
-\Delta u+\alpha\nabla\cdot(u\nabla m)=m(x)u-u^2,\mbox{ in }\Omega
\end{equation}
where $m$ represents the local growth rate (positive in favorable areas and negative in unfavorable areas) and $\alpha\in\R$ represents the rate at which the population moves up the gradient of the growth rate $m(x)$. In that paper the authors studied the above equation associated with, or the no flux boundary condition
\begin{equation}
\partial_{\eta}u-\alpha u\partial_{\eta}m=0,\mbox{ on }\partial\Omega,
\end{equation}
where $\eta$ is the outward unit normal vector at $\partial\Omega$, or the lethal exterior condition
\begin{equation}
u=0,\mbox{ on }\partial\Omega.
\end{equation}

Back to the problem (\ref{AUX}), without the quadratic term $u^2$ and with the same nonlinear function in the
diffusion and advection, namely $-\Delta(u^p)+\vec{\alpha}\cdot \nabla u^p=\lambda u$, was a case analyzed by Pao in \cite{Pao}. In this paper was made change of variables which transforms the original problem into a semilinear elliptic equation. A similar change of variable is not available for (\ref{AUX}) with $p>1$.

Our paper is primarily motivated by \cite{CintraMontenegroSuarez}, in which, with hypotheses similar to those found in \cite{Alves-Delgado-Souto-Suarez}, was possible to obtain some results similar to those of \cite{CintraMontenegroSuarez}, for the case with non-local reaction term. 

It is interesting to note that in our case the presence of the non-local term brought several technical difficulties, moreover of the dependence of $p$ and $|\vec{\alpha}|$ brought an even greater impact than those found in \cite{CintraMontenegroSuarez}, hence some of the our results had considerable modifications in relation to the paper that motivated us. Mainly, in relation to the lack of uniqueness of the problem solution and the strong dependence of $\vec{\alpha}$ and $p$ on the solutions found.

Here, we will consider the class of functions $\mathcal{K}$ which is formed by functions $K:\Omega\times\Omega\rightarrow\R$ such that:

$(K_1)$ $K\in L^{\infty}(\Omega\times\Omega)$ and $K(x,y)\geq0$ for all $x,y\in\Omega$.

$(K_2)$ For all $x\in\Omega$ and $\varepsilon>0$, we have
\begin{equation*}
|[B_{\varepsilon}(x)\times B_{\varepsilon}(x)]\cap E|>0,
\end{equation*}
where $E:=\{(x,y)\in\Omega\times\Omega;K(x,y)>0 \}=K^{-1}(0,+\infty)$.

\begin{remark}\label{remar1}
It is important to note that our assumption $(K_2)$ is equivalent to suppose that: If $w$ is measurable and $\int_{\Omega\times\Omega}K(x,y)|w(y)|^{\gamma}|w(x)|^{2}dxdy=0$, then $w=0$ a.e. in $\Omega$. That is the assumption considered and introduced by \cite{Alves-Delgado-Souto-Suarez}
\end{remark}

Hereafter, we will consider the notation to simply
\begin{equation}
\phi_{u}(x)=\int_{\Omega}K(x,y)|u(y)|^{\gamma}dy
\end{equation}
that is a nonlocal term, and so, the problem $(P)_p$ is rewrite as
$$
\left\{
\begin{array}{lcl}
	-\Delta u+\vec{\alpha}(x)\cdot \nabla (|u|^{p-1}u)=\left(\lambda-\phi_{u}(x) \right)u,\mbox{ in }\Omega\\
	u=0,\mbox{ on }\partial\Omega.
\end{array}
\right.\leqno(P)_p
$$

Before to enunciate the main results, we point out that $\lambda_1[L_{\alpha}]$ denotes the principal eigenvalue of the operator
$L_{\alpha}=-\Delta+\vec{\alpha}\cdot \nabla$ with Dirichlet boundary condition, and more if $\vec{\alpha}=0$, we denotes $\lambda_1:=\lambda_1[L_0]$.

The main theorem establishes a general existence result for the case $\vec{\alpha}\in C^1(\overline{\Omega})^{N}$. It can be stated as
follows.

\begin{theorem}\label{T5}
Suppose that $K\in\mathcal{K}$, $p>1$ and $div\vec{\alpha}(x)=0$. Then, the problem $(P)_p$ has a positive solution if, and only if, $\lambda>\lambda_1$.
\end{theorem}
It will be possible to justify, using the same theorem argument above, but using different operators, the following result:
\begin{corollary}\label{coro1}
Suppose that $K\in\mathcal{K}$ and $\alpha(x)\equiv \vec{\alpha}\in\R^N$. Then, the problem	
$$
\left\{
\begin{array}{lcl}
-\Delta u+\vec{\alpha}\cdot \nabla u=\left(\lambda-\int_{\Omega}K(x,y)u^{\gamma}(y)dy \right)u,\mbox{ in }\Omega\\
u=0,\mbox{ on }\partial\Omega
\end{array}
\right.\leqno(P)_1
$$
has a positive solution if, and only if, $\lambda>\lambda_1[L_\alpha]$.
\end{corollary}

%

After, we will analyze the behavior of the positive solutions with respect to $\vec{\alpha} \in \mathbb{R}^{N}$ and $p\geq1$. 

Denoting by $u_{\lambda,\vec{\alpha}}$ a positive solution of problem $(P)_p$ determined in Theorem \ref{T5}. Our next results analyzes the asymptotic behaviors of $u_{\lambda,\vec{\alpha}}$ with respect to $\vec{\alpha}$.

\begin{theorem}\label{T3}
For each $\lambda>\lambda_1$ and $\vec{\alpha}\in\R^N$, a positive solution $u_{\lambda,\vec{\alpha}}$ of $(P)_p$, with $p>1$, verifies
\begin{equation}\label{eq1}
u_{\lambda,\vec{\alpha}}\rightarrow 0\mbox{ in }L^{\infty}(\overline{\Omega})\mbox{ when }|\vec{\alpha}|\rightarrow\infty
\end{equation}
and
\begin{equation}\label{eq2}
u_{\lambda,\vec{\alpha}}\rightarrow u_{\lambda}\mbox{ in }C^{1}(\overline{\Omega})\mbox{ when }|\vec{\alpha}|\rightarrow0,
\end{equation}
where $u_{\lambda}$ denotes a positive solution of $(P)_p$ with $\vec{\alpha}=0$.	
\end{theorem}

We also analyze the behavior of the solutions when $p\rightarrow1^+$. Thus, we have:

\begin{theorem}\label{T4}
Fix $\lambda>\lambda_1$, $\vec{\alpha}\in\R^N$ and denote $u_{p}$ a positive solution of $(P)_p$. Then, $\lim_{p\rightarrow1^+}u_p=u_*$  in $C^1(\overline{\Omega})$, where $u_*$ is a non-negative weak solution of $(P)_1$.
\end{theorem}

The paper is organized as follows. In Section 2 we present results and basic definitions, which will be essential to reach our goals. In Section 3 we show the existence and nonexistence of positive
solutions, that is, we proved the Theorem \ref{T5}. In Section 4 we analyze the behavior of the positive solutions with respect to $\vec{\alpha}$ and $p$, that is we will prove the theorems \ref{T3} and \ref{T4}.

\textbf{Notations:}
\begin{itemize}
	\item If $\vec{\alpha}=(\alpha_1,\cdots,\alpha_N)\in\R^N$, then $|\vec{\alpha}|=\sqrt{\alpha_1^2+\cdots+\alpha_N^2}$; 
	\item $L^s(\Omega)$, for $1\leq s\leq\infty$, denotes the Lebesgue space with the usual norm denoted by $|u|_s$.
	\item $\|\cdot\|_{C^j(\overline{\Omega})}$ usual norm of $C^j(\overline{\Omega})$, for $j\in\N\cup\{0\}$;
	\item $\|\cdot\|$ usual norm of the Sobolev space $H_0^1(\Omega)$;
	\item $\|\cdot\|_{p,q}$ usual norm of the Sobolev space $W^{p,q}(\Omega)$.
\end{itemize}

\section{Preliminaries}

Initially, we need to fix the result below, of simple proof, but important throughout this work.

\begin{lemma}
Related to the nonlocal term, we have
\begin{description}
	\item[$(\phi_1)$] $t^{\gamma}\phi_w=\phi_{tw}$, for all $w\in L^{\infty}(\Omega)$ and $t>0$;
	\item[$(\phi_2)$] $|\phi_{w}|_{\infty}\leq |K|_{\infty}|\Omega|\cdot|w|_{\infty}^{\gamma}$, for all $w\in L^{\infty}(\Omega)$;
	\item[$(\phi_3)$] $|\phi_{w}-\phi_{v}|_{\infty}\leq |K|_{\infty}|\Omega|\cdot||w|^{\gamma}-|v|^{\gamma}|_{\infty}$, for all $v,w\in L^{\infty}(\Omega)$;
	\item[$(\phi_4)$] $\phi:L^{\infty}(\Omega)\rightarrow L^{\infty}(\Omega)$, $\phi(u):=\phi_u$ is uniformly continuous in $L^{\infty}(\Omega)$.
\end{description}
\end{lemma}
Observe that, the above lemma was introduced in \cite{Alves-Delgado-Souto-Suarez}.

Now, as our arguments to prove Theorem \ref{T5} are based on the classical bifurcation result of Rabinowitz, see \cite{Rabinowitz}, we will recall it. The solution operator $S:C^1(\overline{\Omega})\rightarrow C^1(\overline{\Omega})$ given by
$$
S(v)=w_1\quad\Longleftrightarrow\quad\left\{
\begin{array}{lcl}
	-\Delta w_1=v,\mbox{ in }\Omega\\
	w_1=0,\mbox{ on }\partial\Omega
\end{array}
\right.
$$
is well defined, that is linear and compact operator, moreover verifies
\begin{eqnarray*}
	\|S(v)\|_{C^1(\overline{\Omega})}\leq c\|v\|_{C^1(\overline{\Omega})},\quad\forall v\in C^1(\overline{\Omega})
\end{eqnarray*}
for some $c\in\R$. Related to spectrum of $S$, it is easy to see that
\begin{equation*}
	\sigma(S)=\{\lambda_j^{-1}; \lambda_j \mbox{ is a eigenvalue of the minus Laplacian}\}.
\end{equation*}
On the other hand, define the nonlinear compact operator $G:C^1(\overline{\Omega})\rightarrow C^1(\overline{\Omega})$, given by
$$
G(v)=w_2\quad\Longleftrightarrow\quad\left\{
\begin{array}{lcl}
	-\Delta w_2+\phi_v(x)v+p|v|^{p-1}\vec{\alpha}(x)\cdot\nabla v=0,\mbox{ in }\Omega\\
	w_2=0,\mbox{ on }\partial\Omega
\end{array}
\right.
$$
that is continuous and satisfies
\begin{equation*}
	\|G(v)\|_{C^1(\overline{\Omega})}\leq c(|v|_{\infty}|\phi_{v}|_{\infty}+|v|_{\infty}^{p-1}|\nabla v|_{\infty}),
\end{equation*}
more yet
\begin{equation*}
	\left\|\frac{G(v)}{\|v\|_{C^1(\overline{\Omega})}} \right\|_{C^1(\overline{\Omega})}\leq c(|\phi_{v}|_{\infty}+|v|_{\infty}^{p-1}),\quad\forall v\in C^1(\overline{\Omega}),
\end{equation*}
consequently, if $p>1$, 
\begin{equation*}
	\lim_{v\rightarrow0}\frac{G(v)}{\|v\|_{C^1(\overline{\Omega})}}=0,
\end{equation*}
that is, $G(v)=o(\|v\|_{C^1(\overline{\Omega})})$.
Clearly, under these new notations: $(\lambda,u)$ solves $(P)_p$ if, and only if,
\begin{equation*}
	u=F(\lambda,u):=\lambda S(u)+G(u).
\end{equation*}
Now, as a direct consequence of \cite{Rabinowitz}, considering $E=C^1(\overline{\Omega})$ we have the following result
\begin{theorem}{(Global Bifurcation)}\label{global}
	Let $E$ be a Banach space. Suppose that $S$ is a compact linear operator
	and $\lambda^{-1}\in \sigma(S)$ has odd algebraic multiplicity . If $G$ is a compact operator and
	$$
	\lim_{\|u\|\to 0}\frac{G(u)}{\|u\|}=0,
	$$
	then  the set
	$$
	\Sigma=\overline{\{(\lambda,u)\in\R\times E:u=\lambda S(u)+G(u),u\neq0\}}
	$$
	has a closed connected component $\mathcal{C}=\mathcal{C}_{\lambda}$ such that $(\lambda,0)\in\mathcal{C}$ and
	
	(i) $\mathcal{C}$ is unbounded in $\R\times E$, 
	or
	
	(ii) there exists $\hat{\lambda}\neq \lambda$, such that $(\hat{\lambda},0)\in\mathcal{C}$ and $\hat{\lambda}^{-1}\in\sigma(S)$.
\end{theorem}

\section{Existence and nonexistence of Solutions for $(P)_p$}
The first eigenfunction $\varphi_1$ associated with $\lambda_1$ can be chosen positive. More yet, $\lambda_1^{-1}$ is an eigenvalue with odd multiplicity for $S$. From Theorem \ref{global}, there exists a closed connected component $\mathcal{C}=\mathcal{C}_{\lambda_1}$ of solutions for
$(P)_p$, which verifies $(i)$ or $(ii)$. The next two results follow the ideas from \cite{Alves-Delgado-Souto-Suarez}, particularly Lemma 5 and 6, we leave them in detail to make the text clear to read.

\begin{lemma}
	There exists $\delta>0$ such that if $(\lambda,u)\in\mathcal{C}$ with $|\lambda-\lambda_1|+\|u\|_{C^1(\overline{\Omega})}<\delta$ and $u\neq0$, then $u$ has defined signal, i.e.,
	\begin{equation*}
		u(x)>0\quad\forall x\in\Omega\quad\mbox{ or }\quad u(x)<0\quad\forall x\in\Omega.
	\end{equation*}
\end{lemma}
\begin{proof}
	Let $(u_n)\subset C^1(\overline{\Omega})$ and $\lambda_n\rightarrow \lambda_1$ such that
	\begin{equation*}
		u_n\neq0,\quad \|u_n\|_{C^1(\overline{\Omega})}\rightarrow0\quad\mbox{and}\quad u_n=F(\lambda_n,u_n).
	\end{equation*}
	Consider $w_n=u_n/\|u_n\|_{C^1(\overline{\Omega})}$, by Arzelá-Áscoli Theorem, we have $w_n\rightarrow w$ in $C(\overline{\Omega})$ for some convenient subsequence. Since $S$ can be see as continuous operator from $C(\overline{\Omega})$ to $C^1(\overline{\Omega})$, we get $S(w_n)\rightarrow S(w)$ in $C^1(\overline{\Omega})$. Now, see that
	\begin{equation*}
		u_n=F(\lambda_n,u_n)=\lambda_n S(u_n)+G(u_n)
	\end{equation*}
	and so
	\begin{equation*}
		w_n=\lambda_nS(w_n)+o_n(1),
	\end{equation*}
	consequently $w=\lambda_1S(w)$, that is
	\begin{equation*}
		\left\{
		\begin{array}{lcl}
			-\Delta w=\lambda_1 w,\mbox{ in }\Omega\\
			w=0,\mbox{ on }\partial\Omega.
		\end{array}
		\right.
	\end{equation*}
	Hence, $w\in C^1(\overline{\Omega})$ and $\|w\|_{C^1(\overline{\Omega})}=1$, that is $w\neq0$, by Spectral Theory, we have that
	\begin{equation*}
		w(x)>0\quad\forall x\in\Omega\quad\mbox{ or }\quad w(x)<0\quad\forall x\in\Omega.
	\end{equation*}
	Without loss of generality, we can suppose that $w(x)>0$ for all $x\in\Omega$. Since $w$ is the $C^1(\overline{\Omega})$-limit of $(w_n)$,
	we get $w_n(x)>0$ for all $x\in\Omega$ for $n$ large enough. Therefore, the sign of $u_n$ is the same of $w_n$ for
	$n$ large enough. Concluding the proof.
\end{proof}

It is easy to check that if $(\lambda,u)\in\Sigma$, the pair $(\lambda,-u)$ also is in $\Sigma$. In what follows, we decompose $\mathcal{C}$ into $\mathcal{C}=\mathcal{C}^+\cup \mathcal{C}^-$ where
\begin{equation*}
	\mathcal{C}^+:=\{(\lambda,u)\in\mathcal{C}; u(x)>0,\forall x\in\Omega\}\cup\{(\lambda_1,0)\}
\end{equation*}
and 
\begin{equation*}
	\mathcal{C}^-:=\{(\lambda,u)\in\mathcal{C}; u(x)<0,\forall x\in\Omega\}\cup\{(\lambda_1,0)\}.
\end{equation*}
A simple computation gives that $\mathcal{C}^-=\{(\lambda,u)\in\mathcal{C};(\lambda,-u)\in \mathcal{C}^+ \}$, $\mathcal{C}^+\cap \mathcal{C}^-=\{(\lambda_1,0)\}$ and $\mathcal{C}^+$ is unbounded if, and only if, $\mathcal{C}^-$ is also unbounded.

\begin{lemma}\label{ilim}
	$\mathcal{C}^+$ is unbounded.
\end{lemma}
\begin{proof}
	If $\mathcal{C}^+$ is bounded, we have that $\mathcal{C}$ is bounded too. And so, from Global Bifurcation Theorem, $\mathcal{C}$ contains $(\hat{\lambda},0)$, where $\hat{\lambda}\neq \lambda_1$ and $\hat{\lambda}^{-1}\in\sigma(S)$. As a consequence, we can take $(u_n)\subset C^1(\overline{\Omega})$ and $\lambda_n\rightarrow\hat{\lambda}$ such that 
	\begin{equation*}
		u_n\neq0,\quad \|u_n\|_{C^1(\overline{\Omega})}\rightarrow0\quad\mbox{and}\quad u_n=F(\lambda_n,u_n).
	\end{equation*}
	Considering $w_n=u_n/\|u_n\|_{C^1(\overline{\Omega})}$, we have, repeating previous arguments, $w_n\rightarrow w$ in $C^1(\overline{\Omega})$ with $w\neq0$, moreover
	\begin{equation*}
		\left\{
		\begin{array}{lcl}
			-\Delta w=\hat{\lambda} w,\mbox{ in }\Omega\\
			w=0,\mbox{ on }\partial\Omega.
		\end{array}
		\right.
	\end{equation*}
	showing that $w$ is a eigenfunction related to $\hat{\lambda}$. Since $\hat{\lambda}\neq\lambda_1$, $w$ must change sign. Then, for $n$ large, each $w_n$ must change sign, and the same should be hold for $u_n$. But this is a contradiction, because $(\lambda_n,u_n)\in \mathcal{C}^+$ or $(\lambda_n,u_n)\in \mathcal{C}^-$.
\end{proof}

\textbf{A Priori Estimate:}
From Lemma \ref{ilim}, the connected component $\mathcal{C}^+$ is unbounded. Now, our goal is to show that this component intersects any set of the form $\{\lambda\}\times C^1(\overline{\Omega})$, for $\lambda>\lambda_1$.

\begin{lemma}\label{lem01}
	Suppose that $div \vec{\alpha}(x)=0$ in $\Omega$. For any $\Lambda>0$, there exists $M_{\Lambda}>0$ (that is independent of $\vec{\alpha}$) such that if $(\lambda,u)\in \mathcal{C}^+$ and $\lambda\leq \Lambda$, we get $\|u\|_{C^1(\overline{\Omega})}\leq M_{\Lambda}$.
\end{lemma}
\begin{proof}
	Indeed, arguing by contradiction, if it is not true, there are $(u_n)\subset C^1(\overline{\Omega})$ and $(\lambda_n)\subset[0,\Lambda]$ such that 
	\begin{equation*}
		\|u_n\|_{C^1(\overline{\Omega})}\rightarrow\infty\quad\mbox{and}\quad u_n=F(\lambda_n,u_n).
	\end{equation*}
	Considering $w_n=u_n/\|u_n\|_{C^1(\overline{\Omega})}$, it follows that
	\begin{equation*}
		\int_{\Omega}\nabla w_n\nabla v dx+\int_{\Omega}\phi_{u_n}(x)w_n vdx+p\int_{\Omega}v|u_n|^{p-1}[\vec{\alpha}(x)\cdot\nabla w_n] dx=\lambda_n\int_{\Omega} w_nvdx,\quad\forall v\in C_0^1(\overline{\Omega}).
	\end{equation*}
	Observe that, from same previous arguments, $w_n\rightarrow w$ in $C^1(\overline{\Omega})$ with $w\neq0$. And so, taking $v=\frac{u_n}{\|u_n\|^{\gamma+1}_{C^1(\overline{\Omega})}}$ as a test function, and recalling that $t^\gamma\phi_{u_n}=\phi_{tu_n}$ for all $t>0$, we get
	\begin{equation*}
		\frac{1}{\|u_n\|^\gamma_{C^1(\overline{\Omega})}}\int_{\Omega}|\nabla w_n|^2dx+\int_{\Omega}\phi_{w_n}(x)w_n^2dx+   \frac{p}{\|u_n\|^{\gamma-p+1}_{C^1(\overline{\Omega})}}\int_{\Omega}w_n^p[\vec{\alpha}(x)\cdot\nabla w_n] dx=\frac{\lambda_n}{\|u_n\|^\gamma_{C^1(\overline{\Omega})}}\int_{\Omega}w^2_ndx,\quad\forall n\in\N.
	\end{equation*}
	Now, as $\|u_n\|_{C^1(\overline{\Omega})}\rightarrow\infty$, $(w_n)$ is bounded in $C^1(\overline{\Omega})$ and 
	\begin{equation*}
		p\int_{\Omega}w_n^p[\vec{\alpha}(x)\cdot\nabla w_n] dx=\frac{p}{p+1}\int_{\Omega}\vec{\alpha}(x)\cdot \nabla w_n^{p+1}dx=-\frac{p}{p+1}\int_{\Omega}w_n^{p+1}div(\vec{\alpha}(x))dx+\frac{p}{p+1}\int_{\partial\Omega}w_n^{p+1}(\vec{\alpha}\cdot\eta) d\sigma=0,
	\end{equation*}
	because $div \vec{\alpha}(x)=0$ and $w_n=0$ in $\partial\Omega$. So
	\begin{equation*}
		\lim_{n\rightarrow\infty}\int_{\Omega}\phi_{w_n}(x)w_n^2dx=0.
	\end{equation*}
	From Fatou's Lemma, 
	\begin{equation*}
		0\leq\int_{\Omega}\phi_{w}(x)w^2dx\leq \lim_{n\rightarrow\infty}\int_{\Omega}\phi_{w_n}(x)w_n^2dx\leq0
	\end{equation*}
	and so
	\begin{equation*}
		\int_{\Omega}\phi_{w}(x)w^2dx=0
	\end{equation*}
	that is
	\begin{equation*}
		\int_{\Omega\times\Omega}K(x,y)|w(y)|^{\gamma}|w(x)|^{2}dxdy=0,
	\end{equation*}
	consequently, by $(K_2)$, see also Remark \ref{remar1}, we obtain $w\equiv0$. That is a contradiction. Therefore, the proof is done.
\end{proof}

\begin{lemma}
Suppose that $div \vec{\alpha}(x)=0$ in $\Omega$. The problem $(P)_p$ does not admit a positive solution if $\lambda\leq\lambda_1$.
\end{lemma}
\begin{proof}
	Initially, observe that if $(P)_p$ admits a positive solution for $\lambda\leq\lambda_1$, that is $u>0$ in $\Omega$, we obtain
	\begin{equation*}
		\int_{\Omega}[|\nabla u|^2+pu^p(\vec{\alpha}(x)\cdot\nabla u)]dx=\int_{\Omega}(\lambda u^2-u^2\phi_u(x))dx<\int_{\Omega}\lambda u^2dx
	\end{equation*}
Note that, as $div \vec{\alpha}(x)=0$ and $u=0$ in $\partial\Omega$, we get
\begin{equation}\label{EQI}
	p\int_{\Omega}u^p(\vec{\alpha}(x)\cdot\nabla u) dx=\frac{p}{p+1}\int_{\Omega}\vec{\alpha}(x)\cdot \nabla u_n^{p+1}dx=-\frac{p}{p+1}\int_{\Omega}u_n^{p+1}div(\vec{\alpha}(x))dx+\frac{p}{p+1}\int_{\partial\Omega}u_n^{p+1}(\vec{\alpha}\cdot\eta) d\sigma=0,
\end{equation}
	and so, we have
	\begin{equation*}
		\int_{\Omega}|\nabla u|^2dx<\int_{\Omega}\lambda u^2dx\quad\Rightarrow\quad \lambda>\lambda_1
	\end{equation*}
	that is naturally an absurd. Concluding the proof.
\end{proof}

By previous lemmas, we know the existence and nonexistence of the solution of $(P)_p$, when $p>1$. And so, as a product of previous lemma, we have the proof of Theorem \ref{T5}. 

To finish this section, we observe that, for $p=1$ and $\vec{\alpha}(x)\equiv\vec{\alpha}\in\R^N$ in $(P)_p$, the solution operator $S_0:C^1(\overline{\Omega})\rightarrow C^1(\overline{\Omega})$ given by
$$
S_0(v)=w_1\quad\Longleftrightarrow\quad\left\{
\begin{array}{lcl}
	-\Delta w_1+\vec{\alpha}\cdot w_1=v,\mbox{ in }\Omega\\
	w_1=0,\mbox{ on }\partial\Omega
\end{array}
\right.
$$
is well defined, that is linear and compact operator, moreover verifies
\begin{eqnarray*}
	\|S_0(v)\|_{C^1(\overline{\Omega})}\leq c\|v\|_{C^1(\overline{\Omega})},\quad\forall v\in C^1(\overline{\Omega})
\end{eqnarray*}
for some $c\in\R$. Related to spectrum of $S_0$, it is easy to see that
\begin{equation*}
	\sigma(S_0)=\{(\lambda_j[L_\alpha])^{-1}; \lambda_j[L_\alpha] \mbox{ is a eigenvalue of the $L_\alpha$}\}.
\end{equation*}
On the other hand, define the nonlinear compact operator $G_0:C^1(\overline{\Omega})\rightarrow C^1(\overline{\Omega})$, given by
$$
G_0(v)=w_2\quad\Longleftrightarrow\quad\left\{
\begin{array}{lcl}
	-\Delta w_2+\phi_v(x)v=0,\mbox{ in }\Omega\\
	w_2=0,\mbox{ on }\partial\Omega
\end{array}
\right.
$$
that is continuous and satisfies
\begin{equation*}
	\|G_0(v)\|_{C^1(\overline{\Omega})}\leq c|v|_{\infty}|\phi_{v}|_{\infty},
\end{equation*}
more yet
\begin{equation*}
	\left\|\frac{G_0(v)}{\|v\|_{C^1(\overline{\Omega})}} \right\|_{C^1(\overline{\Omega})}\leq c|\phi_{v}|_{\infty},\quad\forall v\in C^1(\overline{\Omega}),
\end{equation*}
consequently
\begin{equation*}
	\lim_{v\rightarrow0}\frac{G_0(v)}{\|v\|_{C^1(\overline{\Omega})}}=0,
\end{equation*}
that is, $G_0(v)=o(\|v\|_{C^1(\overline{\Omega})})$.
Clearly, under these new notations: $(\lambda,u)$ solves $(P)_1$ if, and only if,
\begin{equation*}
	u=F_0(\lambda,u):=\lambda S_0(u)+G_0(u).
\end{equation*}
Using the similar arguments, lemmas and observations, change $S$ and $G$ by $S_0$ and $G_0$, respectively, to prove the Theorem \ref{T5}, we get the Corollary \ref{coro1}.

\section{Behavior with respect to $\vec{\alpha}$ and $p$}

In this section we analyze the behavior of the positive solutions with respect to $\vec{\alpha}\in\R^N$ and $p>1$. That is, we will prove the Theorems \ref{T3} and \ref{T4}.

First, we will establish  behavior of the positive solutions for $p>1$ with $|\vec{\alpha}|\rightarrow0$ and $|\vec{\alpha}|\rightarrow+\infty$, that is the proof of the Theorem \ref{T3}.

\begin{proof}[Proof of Theorem \ref{T3}]
Firstly, we will to prove (\ref{eq1}). If $|\vec{\alpha}|\rightarrow\infty$ then there exists $\alpha_i\rightarrow\pm\infty$ and denote by $u_{\lambda,\vec{\alpha}}$ a positive solution of $(P)_p$. Suppose, for instance, $\alpha_i\rightarrow\infty$ for some fixed $i$. Let $R>0$ be a constant such that $R-x_i>0$ in $\overline{\Omega}$. Define $\xi(x)=R-x_i>0$. Multiplying $(P)_p$ by $\xi$ and integrating the resulting equation over $\Omega$, we get
\begin{eqnarray*}
-\int_{\partial\Omega}\frac{\partial u_{\lambda,\vec{\alpha}}}{\partial\eta}\xi d\sigma+\int_{\Omega}\nabla u_{\lambda,\vec{\alpha}}\cdot\nabla\xi dx+\int_{\Omega}\xi(\vec{\alpha}\cdot\nabla u_{\lambda,\vec{\alpha}}^p)dx&=&\int_{\Omega}(\lambda u_{\lambda,\vec{\alpha}}\xi-\phi_{u_{\lambda,\vec{\alpha}}}(x)u_{\lambda,\vec{\alpha}}\xi )dx\Longleftrightarrow\\
-\int_{\partial\Omega}\frac{\partial u_{\lambda,\vec{\alpha}}}{\partial\eta}\xi d\sigma-\int_{\Omega}u_{\lambda,\vec{\alpha}}\Delta\xi dx-\int_{\Omega}u_{\lambda,\vec{\alpha}}^p(\vec{\alpha}\cdot\nabla\xi) dx&=&\int_{\Omega}(\lambda u_{\lambda,\vec{\alpha}}\xi-\phi_{u_{\lambda,\vec{\alpha}}}(x)u_{\lambda,\vec{\alpha}}\xi )dx\Longleftrightarrow\\
-\int_{\partial\Omega}\frac{\partial u_{\lambda,\vec{\alpha}}}{\partial\eta}\xi d\sigma&=&\int_{\Omega}(\lambda u_{\lambda,\vec{\alpha}}\xi-\phi_{u_{\lambda,\vec{\alpha}}}(x)u_{\lambda,\vec{\alpha}}\xi )dx-\int_{\Omega}\alpha_iu_{\lambda,\vec{\alpha}}^pdx,
\end{eqnarray*}
where $\eta$ stands for the outward unit normal vector around $\partial\Omega$. Since $\frac{\partial u_{\lambda,\vec{\alpha}}}{\partial\eta}\leq0$, $\lambda>0$ and $\xi>0$, we get
\begin{equation}
0\leq \int_{\Omega}(\lambda u_{\lambda,\vec{\alpha}}|\xi|_{\infty}-\alpha_iu_{\lambda,\vec{\alpha}}^p)dx.
\end{equation}
Since $p>1$, for each $\lambda>\lambda_1$, we obtain
\begin{equation*}
\alpha_i\int_{\Omega}u_{\lambda,\vec{\alpha}}^pdx\leq \lambda|\xi|_{\infty}\int_{\Omega}u_{\lambda,\vec{\alpha}}dx\leq \lambda|\xi|_{\infty}M_{\Lambda}|\Omega|
\end{equation*}
where $M_\Lambda$ is the positive constant that appears in Lemma \ref{lem01}. Since $\alpha_i\rightarrow+\infty$, we conclude that
\begin{equation*}
\int_{\Omega}u_{\lambda,\vec{\alpha}}^pdx\rightarrow0.
\end{equation*}
More yet, as $|u_{\lambda,\vec{\alpha}}|_{\infty}\leq M_{\Lambda}$, we obtain that
\begin{equation}\label{k2}
\int_{\Omega}u_{\lambda,\vec{\alpha}}^tdx\rightarrow0,\quad\forall t\in[1,+\infty).
\end{equation}
Now, for $u_{\lambda,\vec{\alpha}}$ as a test function, we have
\begin{equation}\label{k1}
\int_{\Omega}\left|\nabla u_{\lambda,\vec{\alpha}}\right|^2dx=\int_{\Omega}\nabla u_{\lambda,\vec{\alpha}}\cdot\nabla u_{\lambda,\vec{\alpha}}dx+p\int_{\Omega}u_{\lambda,\vec{\alpha}}^p(\vec{\alpha}\cdot\nabla u_{\lambda,\vec{\alpha}})dx=\int_{\Omega}(\lambda u_{\lambda,\vec{\alpha}}^2-\phi_{u_{\lambda,\vec{\alpha}}}(x)u_{\lambda,\vec{\alpha}}^2)dx\leq \lambda\int_{\Omega}u_{\lambda,\vec{\alpha}}^2dx,
\end{equation}
and so, it is easy to see that $u_{\lambda,\vec{\alpha}}\rightarrow0$ in $H_0^1(\Omega)$.

In this point, we observe that the analogous steps of \cite[Theorem 1.2]{CintraMontenegroSuarez} can be followed and, consequently, it is possible to verifies that 
\begin{equation}
u_{\lambda,\vec{\alpha}}\rightarrow 0\mbox{ in }L^{\infty}(\overline{\Omega}),\mbox{ when }|\vec{\alpha}|\rightarrow\infty.
\end{equation}
With this, we proved the equation (\ref{eq1}) of the Theorem \ref{T3}.

On the other hand, let $\vec{\alpha}_n$ be a sequence such that $|\vec{\alpha}_n|\rightarrow0$ and denote by $u_n=u_{\lambda,\vec{\alpha}_n}$ a positive solution of $(P)_p$ with $\vec{\alpha}$ replaced by $\vec{\alpha}_n$. Observing that $u_{\lambda,\vec{\alpha}_n}$ verifies
\begin{equation*}
-\Delta u_{\lambda,\vec{\alpha}_n}+u_{\lambda,\vec{\alpha}_n}^{p-1}\vec{\alpha}_n\cdot \nabla u_{\lambda,\vec{\alpha}_n}=\left(\lambda-\phi_{u_{\lambda,\vec{\alpha}_n}}(x) \right)u_{\lambda,\vec{\alpha}_n},\mbox{ in }\Omega.
\end{equation*}
Since $u_{\lambda,\vec{\alpha}_n}\leq M_{\lambda}$ by Lemma \ref{lem01}, the Elliptic Regularity (see \cite{GT}) guarantee $\|u_{\lambda,\vec{\alpha}_n}\|_{2,q}\leq C$ for some constant independent of $n$ and $q>1$. As $W^{2,q}(\Omega)\hookrightarrow C^1(\overline{\Omega})$ compactly for $q>N$, then up to a subsequence if necessary, we obtain
\begin{equation*}
u_{\lambda,\vec{\alpha}_n}\rightarrow u_*\mbox{ in }C^1(\overline{\Omega})
\end{equation*}
and, as a consequence, we get
\begin{equation*}
	\left\{
	\begin{array}{lcl}
		-\Delta u_*=(\lambda-\phi_{u_*}(x))u_*,\mbox{ in }\Omega\\
		u_*=0,\mbox{ on }\partial\Omega.
	\end{array}
	\right.
\end{equation*}
From Elliptic Regularity, $u_*$ is a classical solution of the above logistic equation, that is the same logistic problem of \cite{Alves-Delgado-Souto-Suarez}. Now, to complete the proof, we need to prove that $u_*\neq0$. If $u_*=0$, then we define $z_n=u_{\lambda,\vec{\alpha}_n}/|u_{\lambda,\vec{\alpha}_n}|_2$, $n\geq1$. Let $\xi\in H_0^1(\Omega)$ be a test function. Using $\xi$ as a test function, we get
\begin{equation*}
	\int_{\Omega}\nabla u_{\lambda,\vec{\alpha}_n}\nabla\xi dx-\int_{\Omega}u_{\lambda,\vec{\alpha}_n}^p(\vec{\alpha}_n\cdot\nabla\xi)dx=\int_{\Omega}(\lambda u_{\lambda,\vec{\alpha}_n}\xi-\phi_{u_{\lambda,\vec{\alpha}_n}}(x)u_{\lambda,\vec{\alpha}_n}\xi )dx
\end{equation*}
and, dividing by $|u_n|_2$
\begin{equation}\label{e3}
	\int_{\Omega}\nabla z_n\nabla\xi dx-\int_{\Omega}u_{\lambda,\vec{\alpha}_n}^{p-1}z_n(\vec{\alpha}_n\cdot\nabla\xi)dx=\int_{\Omega}(\lambda z_n\xi-\phi_{u_{\lambda,\vec{\alpha}_n}}(x)z_n\xi )dx.
\end{equation}
Taking $\xi=z_n$ as a test function, we conclude
\begin{equation*}
	\int_{\Omega}|\nabla z_n|^2dx\leq C|z_n|_{2}^{2}= C<\infty
\end{equation*}	
consequently, $(z_n)$ is bounded in $H_0^1(\Omega)$. Up to a subsequence if necessary, we have
\begin{equation*}
	z_n\rightharpoonup z\mbox{ in }H_0^1(\Omega)\quad\mbox{and}\quad z_n\rightarrow z\mbox{ in }L^q(\Omega),\quad q\in(2,2^*),
\end{equation*}
with $|z|_2=1$. On the other hand, since
\begin{equation*}
	u_{\lambda,\vec{\alpha}_n}^{p-1}z_n(\vec{\alpha}_n\cdot\nabla\xi)\rightarrow0\mbox{ a.e. in }\Omega\quad\mbox{and}\quad
	|u_{\lambda,\vec{\alpha}_n}^{p-1}z_n(\vec{\alpha}_n\cdot\nabla\xi)|\mbox{ is bounded in }L^1(\Omega)
\end{equation*}
we have, by Lebesgue's Theorem
\begin{equation}
	\int_{\Omega}u_{\lambda,\vec{\alpha}_n}^{p-1}z_n(\vec{\alpha}_n\cdot\nabla\xi)dx\rightarrow0,\mbox{ when }n\rightarrow\infty.
\end{equation}
Analogously,
\begin{equation}
	\int_{\Omega}\phi_{u_n}(x)z_n\xi dx\rightarrow0,\mbox{ when }n\rightarrow\infty.
\end{equation}
Hence, letting $n\rightarrow\infty$ in (\ref{e3}) yields
\begin{equation}
	\int_{\Omega}\nabla z\nabla\xi dx=\lambda\int_{\Omega}z\xi dx,\quad\forall\xi\in H_0^1(\Omega).
\end{equation}
As $z\geq0$ and $z\neq0$, we conclude that $z$ is an eigenfunction of the Laplacian and $\lambda=\lambda_1$, which is an absurd, because $\lambda>\lambda_1$. Concluding the proof.
\end{proof}

For to finish this section, we will to conclude the comparison between cases $p=1$ and $p>1$, we will analyze the behavior of the solutions when $p\rightarrow1^+$.

\begin{proof}[Proof of Theorem \ref{T4}]
Let $(p_n)$ be a sequence such that $p_n>1$ with $p_n\rightarrow1^+$ and denote by $u_n=u_{p_n}$ a positive solution of $(P)_{p_n}$ with $p$ replaced by $p_n$. Similarly, to the previous arguments, up to a subsequence
\begin{equation*}
u_n\rightarrow u_*\mbox{ in }C^1(\overline{\Omega}).
\end{equation*}
Since $u_n$ is a solution of $(P)_{p_n}$, we have
\begin{equation*}
\int_{\Omega}\nabla u_n\cdot\nabla\xi dx+\int_{\Omega}\xi(\vec{\alpha}\cdot \nabla u_n^{p_n})dx=\int_{\Omega}(\lambda u_n\xi-\phi_{u_n}(x)u_n\xi )dx,\quad\forall \xi\in H_0^1(\Omega)
\end{equation*}
and so
\begin{equation*}
\int_{\Omega}\nabla u_n\cdot\nabla\xi dx-\int_{\Omega}u_n^{p_n}(\vec{\alpha}\cdot \nabla\xi )dx=\int_{\Omega}(\lambda u_n\xi-\phi_{u_n}(x)u_n\xi )dx,\quad\forall \xi\in H_0^1(\Omega).
\end{equation*}
For $n\rightarrow+\infty$ in the above equation, we get
\begin{equation*}
\int_{\Omega}\nabla u_*\cdot\nabla\xi dx-\int_{\Omega}u_*(\vec{\alpha}\cdot \nabla\xi )dx=\int_{\Omega}(\lambda u_*\xi-\phi_{u_*}(x)u_*\xi )dx,\quad\forall \xi\in H_0^1(\Omega),
\end{equation*}	
and, consequently
\begin{equation*}
	\int_{\Omega}\nabla u_*\cdot\nabla\xi dx+\int_{\Omega}\xi(\vec{\alpha}\cdot \nabla u_* )dx=\int_{\Omega}(\lambda u_*\xi-\phi_{u_*}(x)u_*\xi )dx,\quad\forall \xi\in H_0^1(\Omega),
\end{equation*}
Thus, $u_*$ is a non-negative weak solution of $(P)_1$.
\end{proof}

\end{document}